\documentclass[12pt]{amsart}
\usepackage{graphicx}
\usepackage{amscd}

\newtheorem{theorem}{Theorem}
\theoremstyle{plain}

\newtheorem{lemma}{Lemma}

\newtheorem{remark}{Remark}

\numberwithin{equation}{section}

\begin{document}
\title[Discontinuous multiplication]{Compactly generated quasitopological homotopy groups with discontinuous
multiplication}
\author{Paul Fabel}
\address{Dept. of Mathematics and Statistics\\
Mississippi State University\\
Drawer MA, Mississippi State, MS 39762}
\email[Fabel]{fabel@ra.edu}
\urladdr{http://www2.msstate.edu/\symbol{126}fabel/}
\date{May 30 2011}
\subjclass{Primary 54G20; Secondary 54B15}
\keywords{}

\begin{abstract}
For each integer $Q\geq 1$ there exists a path connected metric compactum $X$
such that the homotopy group $\pi _{Q}(X,p)$ is compactly generated but not
a topological group (with the quotient topology).
\end{abstract}

\maketitle

\section{Introduction}

Given a space $X$ and a positive integer $Q\geq 1,$ the familiar homotopy
group $\pi _{Q}(X,p)$ becomes a topological space endowed with the quotient
topology induced by the natural surjective map $\Pi
_{Q}:M_{Q}(X,p)\rightarrow \pi _{Q}(X,p).$ ($M_{Q}(X,p)$ denotes the space
of based maps, with the compact open topology, from the $Q-$sphere $S^{Q}$
into $X$).

It is an open problem to understand when $\pi _{Q}(X,p)$ is or is not a
topological group with the standard operations. For example, is $\pi
_{Q}(X,p)$ always a topological group if $Q\geq 2$ (Problem 5.1 \cite{Brazas}%
, abstract \cite{Ghane2})? If $Q\geq 1$ must $\pi _{Q}(X,p)$ be a
topological group if $X$ is a path-connected continuum and $\pi _{Q}(X,p)$
is compactly generated? We answer both questions in the negative via
counterexamples.

In general the topology of $\pi _{Q}(X,p)$ is an invariant of the homotopy
type of the underlying space $X,$ $\pi _{Q}(X,p)$ is a quasitopological
group (i.e. multiplication is continuous separately in each coordinate and
group inversion is continuous), and each map $f:X\rightarrow Y$ induces a
continuous homomorphism $f_{\ast }:\pi _{Q}(X,p)\rightarrow \pi _{Q}(Y,f(p))$
\cite{GH}.

If $X$ has strong local properties (for example if $X$ is locally $n$%
-connected for all $0\leq n\leq Q$) then $\pi _{Q}(X,p)$ is discrete \cite
{Fab1},\cite{Calcut},\cite{GH}, and hence a topological group.

If $Q\geq 2$, the group $\pi _{Q}(X,p)$ is abelian, and this fact has the
capacity to nullify structural pathology present when $Q=1.$ For example
Ghane, Hamed, Mashayekhy, and Mirebrahimi \cite{Ghane2} show if $Q\geq 2$
then $\pi _{Q}(X,p)$ is a topological group if $X$ is the $Q-$dimensional
version of the 1-dimensional Hawaiian earring.

However in general the standard group multiplication in $\pi _{Q}(X,p)$ can
fail to be continuous if $\Pi _{Q}\times \Pi _{Q}:M_{Q}(X,p)\times
M_{Q}(X,p)\rightarrow \pi _{Q}(X,p)\times \pi _{Q}(X,p)$ fails to be a
quotient map. Recent counterexamples respectively of Brazas (Example 4.22 
\cite{Brazas})  and Fabel \cite{Fab3} show $\pi _{1}(X,p)$ fails to be a
topological group if $X$ is the union of large circles parameterized by the
rationals and joined at a common point, or if $X$ is the 1-dimensional
Hawaiian earring.

For the main result, given $Q\geq 1$ we obtain a space $X$ as the union of
convergent line segments $L_{n}\rightarrow L$, joined at the common point $p,
$ with a small $Q-$sphere $S_{n}$ attached to the end of each segment $L_{n}$%
, and this yields the following.

\begin{theorem}
For each $Q\in \{1,2,3,...\}$ there exists a compact path connected metric
space $X$ such that, with the quotient topology, $\pi _{Q}(X,p)$ is
compactly generated and multiplication is discontinuous in $\pi _{Q}(X,p)$.
\end{theorem}

\subsection{Definitions}

If $Y$ is a space and if $A\subset Y$ then the set $A$ is \textbf{closed
under convergent sequences }if $A$ enjoys the following property: If the
sequence $\{a_{1},a_{2},..\}\subset A$ and if $a_{n}\rightarrow a$ then $%
a\in A.$

The space $Y$ is a \textbf{sequential space }if $Y$ enjoys the following
property: If $A\subset Y$ and if $A$ is closed under convergent sequences
then $A$ is a closed subspace of $Y.$

If $Y$ and $Z$ are spaces the surjective map $q:Y\rightarrow Z$ is \textbf{%
quotient map }if for every subset $A\subset Z,$ the set $A$ is closed in $Z$
if and only if the preimage $q^{-1}(A)$ is closed in $Y.$

Given a space $X$ and $p\in X,$ and an integer $Q\geq 1$ let $\pi _{Q}(X,p)$
denote the familiar $Qth$ homotopy group of $X$ based at $p.$

To topologize $\pi _{Q}(X,p)$, let $M_{Q}(X,p)$ denote the space of based
maps $f:(S^{Q},1)\rightarrow (X,p)$ from the $Q$-sphere $S^{Q}$ into $X,$
and impart $M_{Q}(X,p)$ with the compact open topology.

Let $\Pi _{Q}:M_{Q}(X,p)\rightarrow \pi _{Q}(X,p)$ be the canonical quotient
map such that $\Pi _{Q}(f)=\Pi _{Q}(g)$ iff $f$ and $g$ belong to the same
path component of $M_{Q}(X,p),$ and declare $U\subset G$ to be open iff $\Pi
_{Q}^{-1}(U)$ is open in $M_{Q}(X,p).$

A \textbf{Peano continuum} is a compact locally path connected metric space.

If $Q$ is a positive integer a \textbf{closed Q-cell} is any a space
homeomorphic to $[0,1]^{Q},$ and the $\mathbf{Q-}$\textbf{sphere} $S^{Q}$ is
the quotient of $[0,1]^{Q}$ by identifying to a point the $Q-1$ dimensional
boundary $[0,1]^{Q}\backslash (0,1)^{Q}.$ 

\subsection{Basic properties of $\protect\pi _{Q}(X,p)$}

\begin{lemma}
\label{seq}If $X$ is a metrizable space then $M_{Q}(X,p)$ is a metrizable
space, and $\pi _{Q}(X,p)$ is a sequential space.
\end{lemma}

\begin{proof}
Since $X$ is metrizable and since $S^{Q}$ is a compact, the uniform metric
shows $M_{Q}(X,p)$ is metrizable. Moreover since $S^{Q}$ is compact, the
compact open topology coincides with the metric topology of uniform
convergence in $M_{Q}(X,p).$

Suppose $A\subset \pi _{Q}(X,p)$ and suppose $A$ is closed under convergent
sequences. Let $B=\Pi _{Q}^{-1}(A)$. Suppose $b_{n}\rightarrow b$ and $%
b_{n}\in B.$ Then $\Pi _{Q}(b_{n})\rightarrow \Pi _{Q}(b)$ and hence $\Pi
_{Q}(b)\in A.$ Thus $b\in B.$ Thus $B$ is closed under convergent sequences.
Hence $B$ is closed in $M_{Q}(X,p)$ since $M_{Q}(X,p)$ is metrizable. Thus,
since $\Pi _{Q}$ is a quotient map, $A$ is closed in $\pi _{Q}(X,p).$ Hence $%
\pi _{Q}(X,p)$ is a sequential space.
\end{proof}

\begin{lemma}
\label{reps}Suppose $X$ is metrizable and suppose $z_{n}\rightarrow z$ in $%
\pi _{Q}(X,p).$ Then there exists $n_{1}<n_{2}..$ and a convergent sequence $%
\alpha _{n_{k}}\in M_{Q}(X,p)$ such that $\Pi _{Q}(\alpha
_{n_{k}})=z_{n_{k}}.$
\end{lemma}

\begin{proof}
If there exists $N$ such that $z_{n}=z$ for all $n\geq N,$then there exists $%
\alpha \in M_{Q}(X,p)$ such that $\Pi _{Q}(\alpha )=z$ (since $\Pi _{Q}$ is
surjective) and let $\alpha _{n}=\alpha $ for all $n\geq N,$ and thus the
constant sequence $\{\alpha _{n}\}$ converges.

If no such $N$ exists then obtain a subsequence $\{x_{n}\}\subset \{z_{n}\}$
such that $x_{n}\neq z$ for all$\,$\ $n.$ To see that $\{x_{1},x_{2},..\}$
is not closed in $Z$, note $z\notin \{x_{1},x_{2},..\}$ and, since $%
x_{n}\rightarrow z,$ it follows that $z$ is a limit point of the set $%
\{x_{1},x_{2},..\}.$ Thus, since $\Pi _{Q}$ is a quotient map, $\Pi
_{Q}^{-1}\{x_{1},x_{2},..\}$ is not closed in $Y.$ Obtain a limit point $%
\beta \in \overline{\Pi _{Q}^{-1}\{x_{1},x_{2},..\}}\backslash \Pi
_{Q}^{-1}\{x_{1},x_{2},..\}.$ Since $M_{Q}(X,p)$ is metrizable (Lemma \ref
{seq}), there exists a sequence $\beta _{k}\rightarrow \beta $ such that $%
\{\beta _{k}\}\subset $ $\Pi _{Q}^{-1}\{x_{1},x_{2},..\}.$ Moreover
(refining $\{\beta _{k}\}$ if necessary) there exists a subsequence $%
\{m_{1},m_{2},...\}\subset \{1,2,...\}$ such that $\Pi _{Q}(\beta
_{k})=x_{m_{k}}$, and if $k<j$ then $m_{k}<m_{j}.$ Let $x_{m_{k}}=z_{n_{k}}$
and let $\alpha _{n_{k}}=\beta _{k}.$
\end{proof}

\section{Main result}

Fixing a positive integer $Q,$ the goal is to construct a path connected
compact metric space $X,$ such that if $\pi _{Q}(X,p)$ has the quotient
topology, then $\pi _{Q}(X,p)$ is compactly generated but the standard group
multiplication is not continuous in $\pi _{Q}(X,p).$

The idea to construct $X$ is to begin with the cone over a convergent
sequence $\{w,w_{1},w_{2},..\}$ (i.e. we have a sequence of convergent line
segments $L_{n}\rightarrow L$, joined at a common endpoint $p\in L_{n}),$
and then we attach a $Q-$ sphere $S_{n}$ of radius $\frac{1}{10^{n}}$ to the
opposite end of each segment $L_{n}$ at $w_{n}$.

Specifically, for $Q\geq 1$ let $R^{Q+1}$ denote Euclidean space of
dimensions $Q+1$ with Euclidean metric $d.$ Let $p=(0,0,..,0)$ and let $%
w_{n}=(\frac{1}{n},1,0,...0)$ and let $w=(0,1,0,...,0).$ Consider the line
segment $L_{n}=[p,w_{n}]$ and observe for each $n,$ if $i\neq n$ then $%
d(w_{n},w_{i})\geq \frac{1}{n}-\frac{1}{n+1}=\frac{1}{n^{2}+n}.$

Let $L=[p,w]$ and let $\gamma _{n}:L\rightarrow L_{n}$ be the linear
bijection fixing $p$ and observe $\gamma _{n}\rightarrow id|L$ uniformly.

\bigskip Let $c_{n}=(\frac{1}{n},1+\frac{1}{10^{n}},0,..,0).$ Let $S_{n}$
denote the Euclidean $Q$-sphere such that $x\in S_{n}$ iff $d(x,c_{n})=\frac{%
1}{10^{n}}.$ Let $q_{n}=(\frac{1}{n},1+\frac{2}{10^{n}},0,..,0)$ and notice $%
S_{n}\cap L_{n}=\{w_{n}\}.$

Let $X_{n}=\cup _{k=1}^{n}(L_{n}\cup S_{n})$

Let $X=L\cup X_{1}\cup X_{2}....$

Define retracts $R_{n}:X\rightarrow X_{n}$ such that $%
R_{n}(x_{1},x_{2},...x_{n+1})=(x_{1}^{\ast },x_{2},..,x_{n+1})$ with $%
x_{1}^{\ast }$ minimal such that $im(R_{n})\subset X_{n}$.

Notice $R_{n}\rightarrow id_{X}$ uniformly and for all $n$ we have $%
R_{n}R_{n+1}=R_{n}.$

Hence we obtain the natural homomorphism $\phi :\pi _{Q}(X,p)\rightarrow
\lim_{\leftarrow }\pi _{Q}(X_{n},p)$ defined via $\phi ([\alpha
])=([R_{1}(\alpha )],[R_{2}(\alpha )],...).$

For $Q\geq 1$ let $G=\pi _{Q}(X,p)$. Thus if $Q=1$ then $G$ is the free
group on generators $\{x_{1},x_{2},...\}$ and if $Q\geq 2,$ then $G$ is the
free abelian group on generators $\{x_{1},x_{2},...\}$, and let $\ast
:G\times G\rightarrow G$ denote the familiar multiplication.

Elements of $G$ admit a canonical form as maximally reduced words in the
letters $\{x_{1},x_{2},..\}$ in the format $x_{n_{1}}^{k_{1}}\ast
x_{n_{2}}^{k_{2}}\ast ...\ast x_{n_{m}}^{k_{m}}$ with $n_{i}\neq n_{i+1}.$
If $Q\geq 2$ then $G$ is abelian and we can further require $n_{i}<n_{j}$
whenever $i<j.$

Define $l:G\rightarrow \{0,1,2,3,..\}$ such that $l(x_{n_{1}}^{k_{1}}\ast
x_{n_{2}}^{k_{2}}...\ast x_{n_{m}}^{k_{m}})=m.$

Let $G_{N}$ denote the subgroup of $G$ such that if $x_{n_{1}}^{k_{1}}\ast
x_{n_{2}}^{k_{2}}...\ast x_{n_{m}}^{k_{m}}\in G_{N}$ then $n_{i}\leq N$ for
all $i.$

Let $\phi _{N}:G\rightarrow G_{N}$ denote the natural epimorphism such that $%
\phi _{N}(g)$ is the reduced word obtained from $g$ after deleting all
letters $x_{i}$ of index $i>N$.

\begin{lemma}
\label{gt2}The homomorphism $\phi :\pi _{Q}(X,p)\rightarrow \lim_{\leftarrow
}\pi _{Q}(X_{n},p)$ is continuous and one to one, and the space $\pi
_{Q}(X,p)$ is $T_{2}$.
\end{lemma}

\begin{proof}
Recall $G=\pi _{Q}(X,p)$ and note $\pi _{Q}(X_{n},p)$ is canonically
isomorphic to $G_{n}.$ To see that $\phi $ is continuous, first recall in
general a map $\alpha :Y\rightarrow Z$ \cite{GH} induces a continuous
homomorphism $\alpha _{\ast }:\pi _{Q}(Y,y)\rightarrow \pi _{Q}(Z,\alpha (z))
$) and in particular the retractions $R_{n}:X\rightarrow X_{n}$ induce
continuous epimorphisms $R_{n\ast }:G\rightarrow G_{n}.$ By definition $\phi
=(R_{1\ast },R_{2\ast }...)$ and hence $\phi $ is continuous since $%
\lim_{\leftarrow }G_{n}$ enjoys the product topology.

To see that $\phi $ is one to one, suppose $[f]\in \ker \phi .$ Since $%
f:S^{Q}\rightarrow X$ is a map and since $S^{Q}$ is a Peano continuum, then $%
im(f)$ is a Peano continuum, and hence $im(f)$ is locally path connected,
and in particular $im(f)\cap \{q_{1},q_{2},...\}$ is finite.

Obtain $N$ such that $im(f)\cap \{q_{1},q_{2},...\}\subset
\{q_{1},...,q_{N}\}.$ Notice $X_{N}$ is a strong deformation retract of $%
X\backslash \{q_{N+1},q_{N+2},..\}$,$\ $(since for $i\geq N+1$, $%
S_{i}\backslash \{q_{i}\}$ is contractible to $p,$ and we can contract to $p$
simultaneously for $k\in \{1,2,3,...\}$ the subspaces $L\cup L_{N+k}\cup
(S_{N+k}\backslash \{q_{N+k}\}$). In particular, under the strong
deformation retraction collapsing $X\backslash \{q_{N+1},q_{N+2},..\}$ to $%
X_{N}$, $f$ deforms in $X$ to $R_{N}(f),$ and by assumption $R_{N}(f)$
deforms in $X_{N}$ to the constant map (determined by $p$). Hence $f$ is
inessential in $X$ and this proves $\phi $ is one to one.

\bigskip Since $X_{n}$ is locally contractible, $\pi _{Q}(X_{n},p)$ is
discrete \cite{GH}. Hence $\lim_{\leftarrow }\pi _{Q}(X_{n},p)$ is
metrizable and in particular $T_{2}.$ Thus $G$ is $T_{2}$ since $G$ injects
continuously into the $T_{2}$ space $\lim_{\leftarrow }G_{n}$.
\end{proof}

\begin{lemma}
\label{sc}Suppose $\{g,g_{1},g_{2},...\}\subset G.$ Suppose $%
g_{n}\rightarrow g.$ Then $\{l(g_{n})\}$ is bounded and $\phi
_{N}(g_{n})\rightarrow \phi _{N}(g).$
\end{lemma}

\begin{proof}
Suppose $g_{n}\rightarrow g.$ Then $\phi (g_{n})\rightarrow \phi (g)$ in $%
\lim_{\leftarrow }G_{n},$ since $\phi $ is continuous. This means precisely
that for each $N\geq 1,$ the sequence $\phi _{N}(g_{n})\rightarrow \phi
_{N}(g).$

To prove $\{l(g_{n})\}$ is bounded, suppose to obtain a contradiction $%
l\{g_{n}\}$ is not bounded. Select a subsequence $\{y_{n}\}\subset \{g_{n}\}$
such that $l(y_{n})\rightarrow \infty .$ By Lemma \ref{reps} there exists a
subsequence $\{z_{n}\}\subset \{y_{n}\}$ and a convergent sequence of maps $%
\{\alpha _{n}\}\subset M_{Q}(X,p)$ such that $\Pi (\alpha _{n})=z_{n}$. Let $%
z_{n}=x_{s_{1}}^{k_{1}}\ast x_{s_{2}}^{k_{2}}...\ast
x_{s_{m_{n}}}^{k_{m_{n}}}$ and let $Z_{n}\subset
\{q_{s_{1}},q_{s_{1}},..q_{s_{m_{n}}}\}.$ Thus $Z_{n}$ consists of the tops
of the corresponding spheres in $X,$ and hence $Z_{n}\subset im(\alpha
_{n}). $ Since $l(z_{n})\rightarrow \infty $, the maps $\{\alpha _{n}\}$ are
not equicontinuous, contradicting the fact that the convergent sequence $%
\{\alpha _{n}\}$ is equicontinuous.
\end{proof}

\begin{remark}
$\pi _{Q}(X,p)$ is compactly generated. (Select a convergent sequence of
generators $v_{1},v_{2},...\subset M_{Q}(X,p),$such that $[v_{n}]$ generates
the cyclic group $\pi _{Q}(L_{n}\cup S_{n},p))$ and $[v_{n}]\rightarrow e$
and $e$ denotes the identity in $\pi _{Q}(X,p))$
\end{remark}

\begin{theorem}
Multiplication $\ast :\pi _{Q}(X,p)\times \pi _{Q}(X,p)\rightarrow \pi
_{Q}(X,p)$ is not continuous.
\end{theorem}

\begin{proof}
Recall $G=\pi _{Q}(X,p)$ and consider the following doubly indexed subset $%
A\subset G.$ Let $A$ consist of the union of all reduced words of the form $%
x_{n}^{k}\ast x_{k+1}\ast x_{k+2}\ast ...\ast x_{k+n}$, taken over all pairs
of positive integers $n$ and $k$.

To prove that $\ast :G\times G\rightarrow G$ is not continuous, it suffices
to prove that $A$ is closed in $G,$ and that $\ast ^{-1}(A)$ is not closed
in $G\times G.$

To prove $A$ is closed in $G,$ since $G$ is a $T_{2}$ sequential space
(Lemmas \ref{seq} and \ref{gt2}), it suffices to prove every convergent
sequence in $A$ has its limit in $A.$ Suppose $g_{m}\rightarrow g$ and $%
g_{m}\in A$ for all $m.$ Let $g_{m}=x_{n_{m}}^{k_{m}}\ast x_{k_{m}+1}\ast
x_{k_{m}+2}\ast ...\ast x_{k_{m}+n_{m}}$. Notice $l(g_{m})\geq n_{m}.$ Thus
by Lemma \ref{sc} the sequence$\{n_{m}\}$ is bounded. For each $n_{i},$by
Lemma \ref{sc}, the sequence $\phi _{n_{i}}(g_{m})$ is eventually constant
and thus the sequence $\{k_{m}\}$ is bounded (since every subsequence of $%
x_{n_{i}}^{1},x_{n_{i}}^{2},x_{n_{i}}^{3},...$ diverges in $G_{n_{i}}$).
Thus $\{g_{1},g_{2},...\}$ is a finite set and hence (since $G$ is $T_{1}$) $%
\{g_{1},g_{2},...\}$ is closed in $G$. Thus $g\in $ $\{g_{1},g_{2},...\}$
and hence $A$ is closed in $G.$

Let $e\in G$ denote the identity of $G.$ To prove $\ast ^{-1}(A)$ is not
closed in $G\times G$, we will show $(e,e)\notin \ast ^{-1}(A)$ and $(e,e)$
is a limit point of $\ast ^{-1}(A).$ Note $e\ast e=e$ and $e\notin A$ since $%
l(e)=0$ and $l(x)\geq 1$ for all $x\in A.$ Thus $(e,e)\notin \ast ^{-1}(A).$

To see that $(e,e)$ is a limit point of $\ast ^{-1}(A)$ suppose $U\subset G$
is open and suppose $e\in U.$ Let $V=\Pi ^{-1}(U)\subset M_{Q}(X,p).$ First
we show there exists $N$ such that $x_{n}^{k}\in U$ for all $n\geq N$ and
all $k,$ argued as follows.

Obtain a closed $Q-$cell $B\subset S^{Q}$ and recall $w=(0,1,0,...,0).$
Obtain an inessential map $\alpha \in $ $M_{Q}(X,p)$ such that $im(\alpha
)\subset L$ and such that $\alpha ^{-1}(w)=B.$ Let $k_{1},k_{2},...$ be any
sequence of integers and consider the sequence $%
x_{1}^{k_{1}},x_{2}^{k_{2}},...\subset G.$ For each $n\geq 1$ obtain $\alpha
_{n}\in $ $M_{Q}(X,p)$ such that $\Pi (\alpha _{n})=x_{n}^{k_{n}}$, such
that $\alpha _{n}|(S^{Q}\backslash B)=\gamma _{n}(\alpha |(S^{Q}\backslash
B)),$ and such that $\alpha _{n}(B)\subset S_{n}.$ Hence $\alpha
_{n}\rightarrow \alpha $ and thus $\Pi (\alpha _{n})\rightarrow \Pi (\alpha
).$ Hence $x_{n}^{k_{n}}\rightarrow e.$ Since $\{k_{i}\}$ was arbitrary, it
follows there exists $N$ such that $x_{n}^{k}\in U$ for all $n\geq N$ and
all $k.$

Obtain $N$ as above and for each $k\geq 1$ define $v_{k}=x_{k+1}\ast ...\ast
x_{k+N}.$ To see that the sequence $v_{k}\rightarrow e$ select $N$ disjoint
closed $Q$-cells $B_{1},B_{2},...,B_{N}\subset S^{Q}$ (and if $Q=1$ we also
require the closed intervals satisfy $B_{j}<B_{j+1}$ for $1\leq j\leq N-1$).
Construct $\beta :S^{Q}\rightarrow L$ such that $\beta
_{i}^{-1}(w)=B_{1}\cup B_{2}...\cup B_{N}.$ Note $\beta $ is inessential
since $L$ is contractible. For each $i\in \{1,..,N\}$ and for each $k$ let $%
f_{i,k}:S^{Q}\rightarrow X$ satisfy $\Pi (f_{i,k})=x_{k+i},$ and $%
f_{i,k}|(S^{Q}\backslash B_{i})=\gamma _{k}(\beta |(S^{Q}\backslash B_{i})).$
Now let $\beta _{k}=\beta |(S^{1}\backslash (B_{1}\cup ..\cup B_{N}))\cup
f_{1,k}|B_{1}\cup f_{2,k}|B_{2}...\cup f_{N,k}|B_{N}.$ Notice $\beta
_{k}\rightarrow \beta $ uniformly and $\Pi (\beta _{k})=v_{k}$ and $\Pi
(\beta )=e,$ and thus $v_{k}\rightarrow e.$ In particular there exists $K>N$
such that $v_{K}\in U.$ Thus $(x_{N}^{K},v_{K})\in (U,U)$ and $x_{N}^{K}\ast
v_{K}\in A.$ This proves $(e,e)$ is a limit point of $\ast ^{-1}(A),$ and
thus $\ast ^{-1}(A)$ is not closed in $G\times G.$
\end{proof}

\end{document}